\documentclass[preprint, 12pt, english]{elsarticle}
\usepackage{amsmath}
\usepackage{latexsym, amssymb}
\usepackage{amsthm}
\usepackage{color}
\usepackage{txfonts}

\newtheorem{thm}{Theorem}[section] %the resolution could also be [subsection]

\newtheorem{cor}[thm]{Corollary}
\newtheorem{defn}[thm]{Definition}

\newtheorem{lem}[thm]{Lemma}
\newtheorem{prop}[thm]{Proposition}

\newtheorem{rem}[thm]{Remark}

\newcommand\operA[2]{{\if!#2!\operatorname{#1}\else{\operatorname{#1}_{#2}^{\phantom{I}}}\fi}} % To be used within Bdefs. Usage: $\operA{N}{K/F}$ produces $N_{K/F}$; $\operA{N}{}$ produces $N$.
\newcommand\set[1]{\{#1\}}

\newcommand\Lref[1]{{Lemma~\ref{#1}}}%
\newcommand\Pref[1]{{Proposition~\ref{#1}}}%
\newcommand\Cref[1]{{Corollary~\ref{#1}}}%
\newcommand\Rref[1]{{Remark~\ref{#1}}}%
\newcommand\Sref[1]{{Section~\ref{#1}}}%

\def\sub{\subseteq}

\def\tr{{\operatorname{Tr}}}

\def\dim{{\operatorname{dim}}}

\def\Z{\mathbb{Z}}

\def\s{\sigma}

\newcommand\res[2][p]{(#2)_{#1}}
\newcommand\Norm[1][]{\operA{N}{#1}}

\newcommand\mul[1]{{#1^{\times}}} % The multiplicative group
\newcommand{\Trace}[1][]{\if!#1!\operatorname{Tr}\else{\operatorname{Tr}_{#1}^{\phantom{I}}}\fi} % Usage: $\Tr[K/F](a)$.

\long\def\forget#1\forgotten{{}} %
\newcommand\suchthat{{\,:\ \,}}
\newcommand\subjectto{{\,|\ }}

\newcommand\lam{{\lambda}}

\def\({\left(}
\def\){\right)}

\newcommand\isom{{\,\cong\,}}

\newcommand{\lead}{\operatorname{top}}

\newif\iffurther
\furtherfalse
\newif\ifXY % turns XY version on/off
\XYtrue     % Turn it on
\ifXY

\input xy
\input xyidioms.tex
\usepackage{xy}
\xyoption{all} %
\fi % For \ifXY

\newcommand\sg[1]{{\left<#1\right>}}

\usepackage{babel}

\journal{"Journal of Pure and Applied Algebra"}

\begin{document}

\begin{frontmatter}

\title{Kummer Spaces in Cyclic Algebras of Prime Degree}

\author{Adam Chapman}
\ead{adam1chapman@yahoo.com}
\address{Department of Mathematics, Michigan State University, East Lansing, MI 48824}
\author{David J.~Grynkiewicz}
\ead{diambri@hotmail.com}
\address{University of Memphis, Department of Mathematical Sciences, Memphis, TN 38152, USA}
\author{Eliyahu Matzri}
\ead{elimatzri@gmail.com}
\address{Department of Mathematics, Ben-Gurion University, Beer-Sheva, Israel}
\author{Louis H.~Rowen}
\ead{rowen@math.biu.ac.il}
\author{Uzi Vishne}
\address{Department of Mathematics, Bar-Ilan University, Ramat-Gan 5290002, Israel}
\ead{vishne@math.biu.ac.il}

\begin{abstract}
We classify the monomial Kummer subspaces of division cyclic algebras of prime degree $p$, showing that every such space is standard, and in particular the dimension is no greater than $p+1$.
It follows that in a generic cyclic algebra, the dimension of any Kummer subspace is at most $p+1$.
\end{abstract}

\begin{keyword}
Central Simple Algebras, Cyclic Algebras, Kummer Spaces, Generic Algebras, Zero Sum Sequences
\MSC[2010] Primary 16K20; Secondary 11J13
\end{keyword}

\end{frontmatter}

\section{Introduction}

Given an integer $n$ and a central simple $F$-algebra $A$ whose degree is a multiple of $n$, an $n$-Kummer element is an element $v \in A$ satisfying $v^n \in F^\times$ and $v^{n'} \not \in F$ for any $1 \leq n' < n$. (We omit $n$ when it is obvious from the context.)
These elements play an important role in the structure and presentations of these algebras.
For example, in case $\deg(A)=n$ and $F$ is a field of characteristic prime to $n$ containing a primitive $n$th root of unity, $A$ is cyclic if and only if it contains a Kummer element.
Without roots of unity, this equivalence holds when $n$ is prime, but there are counterexamples for general $n$. (See \cite{MRV}.)

A Kummer subspace of $A$ is an $F$-vector subspace $V$ where every $v \in V \setminus \set{0}$ is Kummer.
In case $F$ is of characteristic prime to $n$ containing a primitive $n$th root of unity $\rho$, every cyclic algebra of degree $n$ over $F$ can be presented as
$$F[x,y \subjectto x^n=\alpha,\, y^n=\beta,\, y x y^{-1}=\rho x]$$
for some $\alpha,\beta \in F^\times$.
Assume $A$ is a tensor product of $m$ cyclic algebras of degree $n$ over $F$, and fix a presentation
$$A=\bigotimes_{k=1}^m F[x_k,y_k : x_k^n=\alpha_k, y_k^n=\beta_k, y_k x_k y_k^{-1}=\rho x_k].$$
\begin{defn}
A {\bf{monomial Kummer subspace}} of $A$ (with respect to that fixed presentation) is a Kummer space spanned by elements of the form $\prod_{k=1}^m x_k^{a_k} y_k^{b_k}$ for some $0 \leq a_1,b_1,\dots,a_m,b_m \leq n-1$.
\end{defn}

Assume from now on that $n = p$ is prime. In \cite{Matzri}, the
author made use of the existence of $(m p+1)$-dimensional monomial
Kummer spaces in $A$ to prove that the symbol length of any central
simple $F$-algebras is bounded from above by $p^{r-1}-1$ when $F$ is
a $C_r$ field. We are interested therefore in the maximal possible
dimension of Kummer spaces in general, and monomial Kummer spaces in
particular. Another motivation comes from the generalized Clifford
algebras: if $p+1$ is indeed the maximal dimension of a Kummer space
in a cyclic algebra of degree $p$, as we conjecture, then the
Clifford algebra of a nondegenerate homogeneous polynomial form of
degree $p$ in more than $p+1$ variables cannot have simple images of
degree $p$. (See \cite{ChapVish1} for more information on
generalized Clifford algebras.)

In tensor products of $m$ quaternion algebras, the dimension of Kummer spaces is bounded by $2 m+1$. This is an immediate result of the theory of Clifford algebras of quadratic forms. (See \cite{Lam} for further information.)
The Kummer subspaces of cyclic algebras of degree 3 were classified in \cite{Raczek}, and then in \cite{MV1} and \cite{MV2}, using techniques of composition algebras suggested by \mbox{J.-P.~Tignol}.
The monomial Kummer subspaces of the tensor product of $m$ cyclic algebras of degree $3$ were classified in \cite{Chapman4}, establishing an upper bound of $3 m+1$. This upper bound holds also for non-monomial Kummer spaces in the generic tensor product of $m$ cyclic algebras.

In this paper we study Kummer subspaces in cyclic algebras of degree $p$ for any prime $p$.
We prove that the dimension of monomial Kummer spaces in such algebras is bounded by $p+1$. The proof of this  algebraic fact requires a nontrivial result from elementary number theory (\cite{Index2}, see also \cite{ProjSum}). Finally, we prove in~\Sref{sec:gen} that $p+1$ is the upper bound for the dimension of any Kummer subspace in the generic cyclic algebra.

\section{Kummer subspaces}\label{sec:intro}

Let $p$ be a prime number, $F$ be a field of characteristic either 0 or greater than $p$ containing a primitive $p$th root of unity $\rho$, and $A$ be a cyclic division algebra of degree $p$ over $F$. The variety $X_A$ of all Kummer elements in $A$ is defined by the condition
$s_1 = \cdots = s_{p-1} = 0$,
where $s_i$ are the generic characteristic coefficients.
% When $p = 2$, we obtain $X_A = \set{x \in A \suchthat \tr(x) = 0}$, so $X_A$ is a Kummer space.
We assume that $p \geq 5$.

\subsection{Standard Kummer subspaces}

Let $x \in X_A$. For any $1 \leq k \leq p-1$ we set
$$V_k(x) = Fx + \set{w \in A \suchthat wx = \rho^k xw}.$$

\begin{prop}
Fix $k$.
\begin{enumerate}
\item For every $x \in X_A$, $V_k(x)$ is a Kummer space.
\item The Kummer space $V_k(x)$ determines $x$ up to a scalar factor.
\end{enumerate}
\end{prop}
\begin{proof}
Let $x \in X_A$. By the Skolem-Noether Theorem, there is a Kummer element $y$ such that $yxy^{-1} = \rho x$, and then $V_k(x) = Fx+F[x]y^k$.  For every $c \in F[x]$, $(x+cy^k)^p = x^p + \Norm[{F[x]/F}](c)y^{kp} \in F$, proving that $Fx+F[x]y^k$ is a Kummer space.

Suppose $V_1(x) = V_1(x')$ for $x,x' \in X_A$. As before let $y,y' \in X_A$ be elements such that $yxy^{-1} = \rho x$ and $y' x' y' = \rho x'$. Let $\s$ denote the automorphism of $F[x]$ induced by conjugation by $y$. Since $x', y' \in V_k(x)$, we can write $x' = \alpha x+ w y$ and $y' = \beta x + w' y$ for
$\alpha, \beta \in F$ and $w, w' \in F[x]$. The
condition $y'x' = \rho x'y'$ gives
\begin{equation}\label{mul}\begin{aligned}\alpha \beta x^2+ & (\beta x w  + \rho \alpha x w') y + w' \s(w) y^2
  \\
& =  \rho \alpha\beta x^2 + (\rho^2 \beta x w + \rho \alpha x w')
y+ \rho w \s(w') y^2,
\end{aligned}\end{equation}
which implies $\alpha \beta = 0$. If $\beta \neq 0$ then $\alpha =
0$ implies $w = 0$, which is impossible. Therefore $\beta = 0$,
and the remaining equation is $$w' \s(w) = \rho w \s(w'),$$ from
which it follows that $w  \in F xw'$. But since $x'y' \in V_1(x')
= V_1(x)$, the coefficient of $y^2$ in $x'y'$ must be zero, and hence $w \s(w') = 0$. However $w' \neq 0$, and therefore $x' \in Fx$.

The general argument is obtained by replacing $\rho$ with $\rho^k$.
\end{proof}

\begin{defn}
A Kummer subspace $V \sub A$ is called {\bf{standard}} if it is contained in a space of the form $V_k(x)$ for some Kummer element $x$ and $0 \leq k \leq p-1$.
\end{defn}

\subsection{Criteria for being Kummer}

In order to simplify the expressions, we adopt the following symmetric product notation from \cite{Revoy}:
Given $v_1,\dots,v_t \in A$, let $v_1^{i_1} * \cdots * v_t^{i_t}$ denote the sum of the products of the elements $v_1,\dots,v_1,v_2,\dots,v_2,\dots,v_t,\dots,v_t$ in all possible rearrangements, where each $v_k$ appears exactly $i_k$ times. The superscript $i_k = 1$ is omitted, so for example $x^1 * y^2 = x * y^2$. %% JJ Is it clear now?
The exponentiation notation is used strictly in this sense. We use parentheses when the symmetric product is applied to monomials. For instance, $(x^3)^2 * (y^5)=x^6 y^5+x^3 y^5 x^3+y^5 x^6$.

\begin{prop}\label{crit=p}
Let $v_1,\dots,v_t \in A$. The subspace $V = Fv_1+\cdots+Fv_t$ is Kummer if and only if
$$v_1^{i_1} * \cdots * v_t^{i_t} \in F$$
for every $i_1,\dots,i_t \geq 0$ with $i_1+\cdots+i_t=p$.
\end{prop}
\begin{proof}
By definition $V = Fv_1+\cdots+Fv_t$ is Kummer if and only if $\lam_1 v_1+\cdots+\lam_t v_t$ is Kummer for every $\lam_1,\dots,\lam_t \in F$,
i.e.
$$\sum_{i_1,\dots,i_t} (v_1^{i_1}* \cdots * v_t^{i_t})\lam_1^{i_1} \cdots \lam_t^{i_t} =  (\lam_1 v_1+\cdots+\lam_t v_t)^p \in F.$$
Since $F$ is infinite, the latter is equivalent to having the coefficients
$v_1^{i_1}* \cdots * v_t^{i_t}$ in~$F$.
\end{proof}

\begin{rem}\label{nocomm}
Assume that $Fv+Fv'$ is Kummer where $v$ and $v'$ commute. Then
$v$ and $v'$ are linearly dependent.

Indeed, $pv^{p-1}v' = v^{p-1} * v' \in F$, so $v^{-1}v' \in F$.
\end{rem}

\begin{thm}\label{maxim}
For every $x \in X_A$ and $k$, $V_k(x)$ is maximal with respect to inclusion as a Kummer subspace.
\end{thm}
\begin{proof}
The proof appears in a more general context in \cite{Chapman4}.
As before it suffices to prove that $V_1(x)$ is maximal. Let $y$
be an invertible element such that $yxy^{-1} = \rho x$, so that $V
= V_1(x) = Fx+F[x]y$. Let $z \in A$, and assume $V+Fz$ is
Kummer; we need to show that $z \in V$. Write $z = \sum_{a=0}^{p-1} \sum_{b=0}^{p-1}
\alpha_{a,b}x^ay^b$ for $\alpha_{a,b} \in F$. (We have $\alpha_{0,0}=0$ because $\tr(z)=0$.) For every $a,b$, there
exists some $\ell \not \equiv 0 \pmod{p}$ such that
$x^{a\ell}y^{b\ell} \in V$:
If $b \neq 0$ then take $\ell \equiv b^{-1} \pmod{p}$. Otherwise take $l \equiv a^{-1} \pmod{p}$.
For any $a$ and $b$,
\begin{equation*}\label{mul1}\begin{aligned}\sum_{ij}\alpha_{i,j}
((x^{a\ell}y^{b\ell})^{p-1} * (x^iy^j)) &=
(x^{a\ell}y^{b\ell})^{p-1} * \sum_{ij} \alpha_{i,j} (x^iy^j) \\ &=
(x^{a\ell}y^{b\ell})^{p-1} * z \in F .\end{aligned}\end{equation*}
The coefficient of $x^{a(1-\ell)}y^{b(1-\ell)}$ in this sum is
$$\alpha_{a,b}(x^{a\ell}y^{b\ell})^{p-1} * (x^ay^b) = p
\alpha_{a,b}(x^{a\ell}y^{b\ell})^{p} = p(x^p)^{a\ell}(y^p)^{b\ell}
\alpha_{a,b},$$ so if the monomial $x^ay^b$ is not in $V$, then
$\ell \neq 1$ and necessarily $\alpha_{a,b} = 0$. Consequently $z
\in V$.
\end{proof}

We conclude this section with another criterion for a subspace to be Kummer. We denote the reduced trace by $\tr(\cdot)$.
\begin{lem}\label{basic}
Let $b_1,\dots,b_t \in A$. The subspace $V = Fb_1+\cdots+Fb_t$ is Kummer if and only if $\tr(b_1^{i_1} * \cdots * b_t^{i_t}) = 0$ for every $i_1,\dots,i_t \geq 0$ satisfying $i_1+\cdots+i_t < p$.
\end{lem}
\begin{proof}
An element $x \in A$ is Kummer if and only if $\tr(x^i) = 0$ for every $i = 1,\dots,p-1$. The rest of the proof is the same as in \Pref{crit=p}.
\end{proof}

The usefulness of the second criterion is emphasized in the following observation:
\begin{lem}\label{nonzcoef}
Fix a presentation $A=F[x,y \subjectto x^p=\alpha, y^p=\beta, y x y^{-1}=\rho x]$. Let $v_1,\dots,v_t$ be monomials. Then, for every $i_1,\dots,i_t \geq 0$ with $i_1+\cdots+i_t < p$, $v_1^{i_1} * \cdots * v_t^{i_t}$ is a nonzero multiple of $v_1^{i_1}\cdots v_t^{i_t}$.
\end{lem}
\begin{proof}
Since each $v_i$ is monomial, the multiplicative commutator of every $v_j,v_{j'}$ is a power of $\rho$. Therefore, each summand in the symmetric product $v_1^{i_1} * \cdots * v_t^{i_t}$ is a multiple of $v_1^{i_1}\cdots v_t^{i_t}$ by some power of $\rho$, and when we write
$$v_1^{i_1} * \cdots * v_t^{i_t} = c \cdot v_1^{i_1}\cdots v_t^{i_t},$$
we have that $c \in \Z[\rho]$ (more precisely in the image of $\Z[\rho]$ in $F$).

Modulo $1-\rho$,\ \ $c$ is equivalent to the number of summands, namely $c \equiv \binom{i_1+\cdots+i_t}{i_1,\dots,i_t}$ in the quotient $\Z[\rho]/(1-\rho)\Z[\rho] \isom \Z/p\Z$. But the multinomial coefficient is nonzero modulo $p$ because $(i_1+\cdots+i_t)!$ is prime to $p$.
\end{proof}

\section{Monomial Kummer subspaces}\label{sec:mon}

Fix a presentation
$$A=F[x,y \subjectto x^p=\alpha, y^p=\beta, y x y^{-1}=\rho x].$$
Recall that a Kummer subspace $V \sub A$ is {\bf{monomial}} if it is spanned by elements of the form $x^i y^j$. In this section we classify monomial Kummer subspaces, showing that they are all standard.

\begin{lem}
A subspace $V \sub A$ is monomial if and only if it is invariant under conjugation by $x$ and $y$.
\end{lem}
\begin{proof}
A monomial subspace is obviously invariant.
Assume $V$ is invariant under conjugation by $x$ and $y$. Let $v \in
A$. Write $$v=f_0+f_1 y+\dots+f_{p-1} y^{p-1}$$ where
$f_0,\dots,f_{p-1} \in F[x]$. Then $$\sum _{i=0}^{p-1}\rho ^{-ij}
f_i y^i =x^jvx^{-j} \in V$$ for $0 \le j < p,$ implying by a
standard Vandermonde argument (based on the fact that the matrix
$(\rho ^{ij}) : 0 \le i,j < p$ is invertible) % Example 0.9 of Grad Alg text (Rowen)
that    $f_i y^i \in
V$ for each $0 \leq i \leq p-1$. Now writing $f_i = \sum _j
\alpha_{i,j} x^j$ for $\alpha_{i,j} \in F$  and conjugating by $y$
yields by the same argument that each $\alpha_{i,j} x^j y^i \in V.$
Going over all the elements in $V$, one obtains a set of monomials
in $V$ spanning $V$.
\end{proof}

\subsection{$3$-dimensional Kummer spaces}\label{ss:mon3}

We commence with Kummer spaces of dimension $3$.
\begin{rem}\label{thoseare}
In the following cases, the space
$$U = Fx+Fy+Fx^ay^b$$
is Kummer: $a = 1$, $b = 1$, $a+b \equiv 0 \pmod{p}$ and $a+b \equiv 1 \pmod{p}$. In all of these cases $U$ is standard:
\begin{eqnarray*}
Fx + Fy + Fxy^b %= (F+Fy^b)x + Fy
& \sub &  Fy+F[y]x ; \\
Fx + Fy + Fx^ay %= (F+Fx^a)y + Fx
& \sub &  Fx+F[x]y; \\
Fx + Fy + Fx^ay^{-a} %= (F + Fxy^{-1})y + F(xy^{-1})^a
& \sub &  F(xy^{-1})^a + F[xy^{-1}]y; \\
Fx + Fy + Fx^ay^{1-a} %= (F + Fxy^{-1} + F(xy^{-1})^a)y
& \sub &  F[xy^{-1}]y.
\end{eqnarray*}
\end{rem}

For every integer $a \in \Z$, let $\res{a}$ denote the unique residue $\res{a} \equiv a \pmod{p}$ such that $0\leq \res{a} < p$.
\begin{prop}\label{dim3}
Let $U = Fx+Fy+Fx^ay^b$. Then $U$ is {\emph{not}} Kummer if and only if there is some $k$, invertible modulo $p$, such that $\res{ka}+\res{kb}+\res{-k} < p$.
\end{prop}
\begin{proof}
For every positive $i,j,k$ with $i+j+k<p$, write $x^i * y^j * (x^ay^b)^k = c_{ijk} x^{i+ka}y^{j+kb}$ for a suitable constant $c_{ijk} \in \Z[\rho]$, which is nonzero by \Lref{nonzcoef}.
%this definition, $c_{ijk}$ is a sum of $\binom{i+j+k}{i,j,k}$ powers of $\rho$, so modulo $1-\rho$, $c_{ijk} \equiv \binom{i+j+k}{i,j,k} \pmod{p}$. When $i+j+k<p$, $c_{ijk} \not \equiv 0 \pmod{p}$, which implies that $c_{ijk} \neq 0$.

By \Lref{basic}, $U$ is not Kummer if and only if there are some positive $i,j,k$ with $i+j+k<p$ such that $$c_{ijk}\tr(x^{i+ka}y^{j+kb}) = \tr(x^i * y^j * (x^ay^b)^k) \neq  0.$$ But the reduced trace of a non-central monomial is zero, so $U$ is not Kummer if and only if there are positive $i,j,k$ with $i+j+k<p$ for which $x^{i+ka}y^{j+kb} \in F$, namely $i \equiv -k a$, $j \equiv - kb$.
\end{proof}

Let $\sg{z}$ denote $\mul{F} z$ for any $z \in X_A$. Consider the subgroup $G$ of $\mul{A}/\mul{F}$ generated by $\mul{F}x$ and $\mul{F}y$. Clearly $G \isom \Z/p\Z \times \Z/p\Z$.
\begin{prop}\label{class3}
Given $z_1,z_2,z_3 \in X_A$, the space
$$U = F z_1 + F z_2 + F z_3$$
is Kummer if and only if:
 \begin{enumerate}
 \item there are no $i \neq j$ with $z_j \in \sg{z_i}$, and
 \item either $\sg{z_i z_j^{-1}} = \sg{z_k}$ for some permutation $\set{i,j,k}$ of $\set{1,2,3}$, or $\sg{z_1z_2^{-1}} = \sg{z_2z_3^{-1}} = \sg{z_3z_1^{-1}}$.
     \end{enumerate}
\end{prop}
\begin{proof}
The first requirement follows from \Pref{nocomm}. Therefore, we may assume that any two of $\sg{z_1},\sg{z_2},\sg{z_3}$ generate $G$. By changing generators and the choice of root of unity, we may assume $z_1 = x$ and $z_2 = y$. The condition then translates to: $U = Fx+Fy+Fx^ay^b$ is Kummer if and only if one of the following holds:
\begin{enumerate}
\item[1.] $\sg{xy^{-1}} = \sg{x^ay^b}$, or equivalently, $a+b\equiv 0 \pmod{p}$;
\item[2.] $\sg{y} = \sg{x^{a-1}y^{b}}$, or equivalently, $a = 1$;
\item[3.] $\sg{x} = \sg{x^ay^{b-1}}$, or equivalently, $b = 1 $;
\item[4.] $\sg{xy^{-1}} = \sg{x^{a-1}y^b} = \sg{x^ay^{b-1}}$, or equivalently, $a+b\equiv 1 \pmod{p}$.
\end{enumerate}
These are the cases listed in \Rref{thoseare} as Kummer subspaces, and it remains to show that $U$ is not Kummer in any other case. Let $a,b \in \Z/p\Z$ be numbers such that $x^ay^b$ is not in $\sg{x}$ or $\sg{y}$, and such that we are not in any of the four cases listed above. Consider the vector $(a,b,-1)$ over $\Z/p\Z$. It has no zero entries, no sum of two entries is zero, and $a+b-1$ is nonzero. It was shown in \cite{Index2} that there is some invertible $k \in \Z/p\Z$ such that $\res{ak}+\res{bk}+\res{-k}<p$, so $U$ is not Kummer by \Pref{dim3}.
\end{proof}

\subsection{Kummer spaces of dimension greater than 3}

\begin{lem}\label{dim4}
The space $U = Fx+Fy+Fx^ay^b+Fx^cy^d$ is not Kummer if there are integers $m,\ell$ such that
$0 < \res{a m+c\ell} + \res{b m+d\ell} + \res{-m} + \res{-\ell} < p.$
\end{lem}

\begin{proof}
Assume such integers exist. Then $w = x^{\res{a m+c\ell}} * y^{\res{b m+d\ell}} * (x^a y^b)^{\res{-m}} * (x^c y^d)^{\res{-\ell}}$ is a nonzero multiple of the scalar $x^{\res{a m+c\ell}}y^{\res{b m+d\ell}}(x^a y^b)^{\res{-m}}(x^c y^d)^{\res{-\ell}}$ by \Lref{nonzcoef}, so that $\tr(w) \neq 0$, and \Lref{basic} shows that $U$ is not Kummer.
\end{proof}

\begin{thm}\label{greatdim}
Every monomial Kummer space of dimension greater than $3$ whose basis contains $x$ and $y$ is contained in either
\begin{itemize}
\item $V_1(x)=F[x] y+F x$, or
\item $V_{p-1}(y)=F[y] x+F y$, or
\item $V_k(v)=F[v] x+F v$ where $v=(x y^{-1})^k $ for some $1 \leq k \leq p-1$.
\end{itemize}
\end{thm}

\begin{proof}
Let $\set{x,y,u,w,\dots}$ be the basis.
Assume $u=x y^k$ and $w=x^i y$ for some $1 \leq k,i \leq p-1$.
Since $\sg{w}=\sg{u^{k^{-1}} x^{i-k^{-1}}}$, one of the following holds: $k=1$, $i-k^{-1} \equiv 1 \pmod{p}$, $k^{-1}+i-k^{-1}=i \equiv 0 \pmod{p}$ or $i \equiv 1 \pmod{p}$.
The case of $i \equiv 0$ is out of the question.
If $i,k \neq 1$ then $i-k^{-1} \equiv 1 \pmod{p}$, which means that $k i-1 \equiv k \pmod{p}$.
For similar reasons we obtain from $\sg{u}=\sg{w^{i^{-1}} y^{k-i^{-1}}}$ that $k i-1 \equiv i \pmod{p}$.
Therefore, $k=i$.
However, in this case the condition from Lemma \ref{dim4} for not being Kummer holds for this space: take $m = -1$ and $\ell = 2i - 2$ if $i  \leq \frac{p-1}{2}$, and
$m = -1$ and $\ell = -1$ if $\frac{p+1}{2} \leq i$.

Assume $u=x^i y$ and $w=x^{-k} y^k$ for some $1 \leq k,i \leq p-1$.
Since $\sg{w}=\sg{u^{k} x^{-k-i k}}$, either $k=1$, $-k-i k \equiv 1 \pmod{p}$, $k-k-i k=-i k \equiv 0 \pmod{p}$ or $-i k \equiv 1 \pmod{p}$.
In case $k=1$, $w,u \in F[x] y$.
Assume $k \neq 1$.
The case of $-i k \equiv 0 \pmod{p}$ is impossible.
From $\sg{w}=\sg{u^{-k i^{-1}} y^{k+k i^{-1}}}$ we obtain that either $-k i^{-1} \equiv 1 \pmod{p}$, $k+k i^{-1} \equiv 1 \pmod{p}$, $k=0$ or $k=1$.
The two last options are out of the question.
The first option implies $i \equiv -k \pmod{p}$ and the second $k i+k \equiv i \pmod{p}$.
If $i \equiv -k \pmod{p}$ and $-i k \equiv 1 \pmod{p}$ then $k^2 \equiv 1 \pmod{p}$ which means $k=p-1$.
In this case, $w,u \in F[y] x$.
If $i \equiv -k \pmod{p}$ and $-k-i k \equiv 1 \pmod{p}$ then $k^2-k-1 \equiv 0 \pmod{p}$.
However, in this case the condition from Lemma \ref{dim4} for not being Kummer holds for this space: take $m=-1$ and $\ell=k+1$ if $\frac{p+1}{2} \leq k$, and $m=2-k$ and $\ell=-1$ if $k \leq \frac{p-1}{2}$.
If $k i+k \equiv i \pmod{p}$ and $-k-i k \equiv 1 \pmod{p}$ then $i=p-1$. In this case, $w$ commutes with $v$, contradiction.
If $k i+k \equiv i \pmod{p}$ and $-i k \equiv 1 \pmod{p}$ then $k \equiv i+1 \pmod{p}$.
In this case, the condition from Lemma \ref{dim4} for not being Kummer holds for this space: take $m=\ell=-1$.

Assume $u=x^i y$ and $x^{-k} y^{k+1}$ for some $1 \leq i \leq p-1$ and $1 \leq k \leq p-2$.
Since $\sg{w}=\sg{u^{k+1} x^{-k-i (k+1)}}$, either $k=0$, $-k-i (k+1) \equiv 1 \pmod{p}$, $1-i (k+1) \equiv 0 \pmod{p}$ or $-i (k+1) \equiv 0 \pmod{p}$.
The first option is impossible. The last option implies $k=p-1$, contradiction.
From $\sg{w}=\sg{u^{-k i^{-1}} y^{k+1+k i^{-1}}}$, either $-k i^{-1} \equiv 1 \pmod{p}$, $k+1+k i^{-1} \equiv 1 \pmod{p}$, $k=p-1$ or $k=0$.
The last two options are impossible. The first option translates to $i \equiv -k \pmod{p}$ and the second to $k (i+1) \equiv 0 \pmod{p}$, i.e. $i=p-1$.
If $i=p-1$ then $w,u \in F[u] x+F u$.
Assume $i \equiv -k \pmod{p}$ and $i \neq p-1$.
If $-k-i (k+1) \equiv 1 \pmod{p}$ then $k^2 \equiv 1 \pmod{p}$, which means $k=p-1$ or $k=1$, contradiction.
If $1-i (k+1) \equiv 0 \pmod{p}$ then $k^2+k+1 \equiv 0 \pmod{p}$. In this case, however, the condition from Lemma \ref{dim4} for not being Kummer holds for this space: take $m=-1$ and $\ell=2+k$.

If $u=x^{-k} y^k$ and $w=x^{-i} y^{i}$ then $u$ and $w$ commute, contradiction.
If $u=x^{-k} y^k$ and $w=x^{-i} y^{i+1}$ then $w,u \in F[u] w+F u$.

In conclusion, if the basis contains a monomial of the form $x^i y$ with $2 \leq i \leq p-2$ then all the other basic elements must belong to $F[x] y$.
Similarly, if the basis contains a monomial of the form $x y^k$ with $2 \leq k \leq p-2$ then all the other basic elements must belong to $F[y] x$.
If the basis contains a monomial of the form $x^{-k} y^k$ with $2 \leq k \leq p-2$ then all the other basic elements must belong to $F[x^{-k} y^k] x+x^{-k} y^k$.
If the basis contains the monomial $x y$ then all the other basic elements must belong to $F[x] y+F[y] x$.
If the basis contains the monomial $x^{p-1} y$ then all the other basic elements belong to $F[x] y+F[x^{-1} y] x$.
If the basis contains the monomial $x y^{p-1}$ then all the other basic elements belong to $F[y] x+F[x^{-1} y] x$.
The monomial Kummer spaces that do not contain elements of the forms $x^i y,x y^k, x^{-k} y^k$ are contained in $F[x^{-1} y] x$.
The statement follows immediately.
\end{proof}

All the arguments in this section can be repeated for any pair of monomials in the basis of a monomial Kummer space, not just $x$ and $y$.
Therefore we obtain the following:

\begin{cor}\label{greatdimcor}
Every monomial Kummer space is standard. In particular, the dimension of any monomial Kummer space is at most $p+1$.
\end{cor}

\section{Kummer subspaces in the generic cyclic algebra of degree $p$}\label{sec:gen}

In this section we consider maximal Kummer subspaces in the generic cyclic algebra of degree $p$, and show that their dimension is at most $p+1$.

The generic cyclic algebra is constructed as follows, when the
ground field $F$  has characteristic prime to $p$ and contains $p$th
roots of unity: Let $$T = F[X,Y\suchthat YX = \rho XY]$$ denote the
quantum plane with the commutator specialized to $\rho$. Let $\alpha
= X^p$ and $\beta = Y^p$. Localizing at the center $T_0 =
F[X^p,Y^p]$, we obtain the division algebra $D = (T_0 \setminus
\set{0})^{-1}T$, which is cyclic over its own center $K = q(T_0) =
F(\alpha,\beta)$. This algebra is generic as a cyclic algebra, as we
can specialize $X,Y$ to a standard pair of generators in any cyclic
division algebra over $F$.

\forget
Let $K = F(\alpha,\beta)$. Let
To construct the generic cyclic algebra we start with a split algebra $S=(1,1)_{p,F}=F[x,y | x^p=1;y^p=1; yx=\rho xy]$ over a field $F$.
We then extend scalars to $K=F(\alpha, \beta)$ to get $S\otimes K=K[x,y | x^p=1;y^p=1; yx=\rho xy]$. Now we consider a subring $T=F[X=\alpha x;Y=\beta y]$ with center $F[\alpha ^p,\beta ^p]$ and define $D=T\otimes F(\alpha ^p,\beta ^p)$ as the generic cyclic algebra.
The name generic is explained by the fact that every cyclic algebra $A=(a,b)_{p,F}$ is a specialization of $D$ (sending $X,Y$ to the standard generators $x,y$ of $A$). The study of $D$ will be carried out through the subring $T$.
\forgotten

Every element of $T$ can be written uniquely as a polynomial of the form $\sum_{i,j=0}^{N} \alpha_{ij}X^iY^j$ with coefficients $\alpha_{ij} \in F$. This induces a natural $\Z\times\Z$-grading where the homogeneous components are monomials in $X,Y$ over $F$.
We order $\Z\times \Z$ lexicographically and denote by $\deg(t)$ the degree of $t$, and by $\lead(t)$ the leading monomial of $t\in T$, namely when $t = \sum a_{i,j}X^iY^j$, $\deg(t) = (i,j)$ and $\lead(t)=a_{i,j}X^iY^j$ with $(i,j)$ maximal such that $a_{i,j}\neq 0$.
\begin{rem}
For every $t_1,t_2 \in T$, $\lead(t_1t_2)=\lead(t_1)\lead(t_2)$.
\end{rem}

Now let $V\sub D$ be a Kummer subspace. Clearing denominators in a basis of $V$ over the center, we may write $V = K\cdot V_0$ where $V_0 \sub T$ is a (finite) module over~$T_0$.

\begin{prop}\label{MP}
Let $v_1,\dots,v_k \in T$. If $V=Fv_1+\cdots+Fv_k \subset D$ is Kummer then so is $\hat{V}= F \lead(v_1) + \cdots + F\lead(v_k)$.
\end{prop}
\begin{proof}
By \Lref{basic}, we only need to show that $\tr(\lead(v_1)^{i_1} *
\cdots * \lead(v_k)^{i_k}) = 0$ for every $i_1,\dots,i_k \geq 0$
with $i_1+\cdots+i_k < p$. Since $\lead(v_1)^{i_1} * \cdots *
\lead(v_k)^{i_k}$ is a multiple of $\lead(v_1)^{i_1} \cdots
\lead(v_k)^{i_k}$, we need only show that $\tr(\lead(v_1)^{i_1}
\cdots \lead(v_k)^{i_k}) = 0$.

But the fact that $V$ is Kummer implies $\tr(v_1^{i_1} * \cdots *
v_k^{i_k}) = 0$, which in particular implies that its coefficient of
degree $i_1\deg(v_1)+\cdots+i_k\deg(v_k)$ is zero. This coefficient
is $\tr(\lead(v_1)^{i_1} * \cdots * \lead(v_k)^{i_k})$,  a nonzero
multiple of $\tr(\lead(v_1)^{i_1} \cdots \lead(v_k)^{i_k})$, by
\Lref{nonzcoef}, so $\tr(\lead(v_1)^{i_1} \cdots \lead(v_k)^{i_k}) =
0$ as desired.
\end{proof}

\begin{rem}\label{normalize}
A Kummer subspace $V\subset D$ has a basis contained in $T$ and with distinct degrees modulo $p\Z \times p\Z$.
\end{rem}
\begin{proof}
%The proof appears in a more general context in \cite[Theorem 4.1]{Chapman4}.
Clearing denominators we may assume the basis elements $v_1,\dots,v_k$ are in $T$.
% Thus, each $v_r$ can be written as $\sum_{i=0}^{p-1} \sum_{j=0}^{p-1} c_{i,j}^{(r)} X^i Y^j$ where $c_{i,j}^{(r)} \in F[\alpha,\beta]$.
Fix an arbitrary linear order on $\Z/p\Z \times \Z/p\Z$. The degree
of $v \in T$ now denotes the maximal $(i,j)$ for which $v$ has a
monomial in $F[\alpha,\beta]X^iY^j$. If some $v_r,v_s$ ($r<s$) have
the same degree, let $c_r$ and $c_s$ denote the coefficients of the
leading monomials, and replace $v_s$ by $c_r v_s - c_s v_r$. This
does not change $\sum K v_i$. Moreover the resulting vector cannot
be zero because of the independence over~$K$. And finally the degree
vector of $v_1,\dots,v_k$ has been lexicographically reduced,
establishing that the process is finite, culminating in the desired
basis.
\end{proof}

\begin{thm}\label{ingen}
The dimension of a Kummer subspace of $D$ is at most $p+1$.
\end{thm}
\begin{proof}
Let $V \sub D$ be a Kummer subspace. By \Rref{normalize} there is a
basis $v_1,\dots,v_k$ of $V$ whose elements are in $T$ and have
distinct degrees modulo $p\Z \times p\Z$. The space
 $\hat{V} = F\lead(v_1)+\cdots+F\lead(v_k)$ is clearly monomial (with respect to~$X,Y$) and is Kummer by \Pref{MP}.
 By \Cref{greatdimcor}, $\dim(V) = \dim(\hat{V}) \leq p+1$.
\end{proof}

\section*{Acknowledgements}
This work is the outcome of the noncommutative algebra group meetings held at Bar-Ilan University during the academic year 2012-2013.
Chapman, Rowen and Vishne were partially supported by BSF grant 2010/149.
Grynkiewicz was partially supported by FWF grant P21576-N18.
Matzri was partially supported by a Kreitman fellowship and by the Israel Science Foundation (grant No. 152/13).

\section*{Bibliography}
\bibliographystyle{amsalpha}
\bibliography{bibfile}
\end{document}